\documentclass{amsart}
\usepackage{textcomp}
\usepackage{indentfirst}
\usepackage{amsmath}
\usepackage{amsthm}
\usepackage{amsfonts}
\usepackage{amssymb}
\usepackage{mathrsfs}
\usepackage{latexsym}
\usepackage{amssymb}
\usepackage{amsfonts}
\usepackage{textcomp}
\usepackage{xcolor}
\usepackage{indentfirst}
\usepackage{mathrsfs}
\usepackage[all]{xy}
\usepackage{bbm}
\usepackage[T1]{fontenc}
\usepackage{enumerate}
\newtheorem{theorem}{Theorem}[section]
\newtheorem{lemma}[theorem]{Lemma}
\newtheorem{proposition}[theorem]{Proposition}
\newtheorem{corollary}[theorem]{Corollary}

\theoremstyle{definition}

\newtheorem{remark}[theorem]{Remark}
\newtheorem{definition}[theorem]{Definition}

\numberwithin{equation}{section}
\newcommand{\h}{\mathcal{H}}
\newcommand{\ka}{\mathcal{K}}

\newcommand{\bhk}{\mathcal{B}(\h,\ka)}

\def\ata{A^{\theta,\alpha}}

\def\nk0t{\|\tilde k_0^\theta\|^{-2}}
\def\kda{K_\alpha}
\def\kdt{K_\theta}

\def\tto{truncated Toeplitz operator}
\def\dtto{dual truncated Toeplitz operator}

\def\b1{\mathcal{B}_1(\kdt,\kda)}

\newcommand{\dc}{P^- D_{|\theta H^2}}
\newcommand{\dd}{P^- D_{| H^2_-}}
\newcommand{\ak}{\mathbb A_{1/k}}
\newcommand{\dtha}{P_{\alpha H^2}D_{|\theta H^2}}
\newcommand{\s}{{\mathcal S}}
\newcommand{\daa}{P_{\alpha H^2}D_{|\theta H^2}}
\newcommand{\dba}{P_{\alpha H^2}D_{|H^2_-}}
\title[Shift invariance and reflexivity]{Shift invariance and reflexivity  of compressions of multiplication operators}

\author[M. C. C\^amara, K. Kli\'s--Garlicka, B. \L anucha, and M. Ptak]{M. Cristina C\^amara, Kamila Kli\'s--Garlicka, Bartosz \L anucha, and Marek Ptak}

\address{M. Cristina C\^{a}mara, Center for Mathematical Analysis, Geometry and Dynamical Systems\\ Mathematics
Department, Instituto Superior T\'ecnico, Universidade de Lisboa\\ Av. Rovisco Pais, 1049-
001 Lisboa, Portugal.}\email{ccamara@math.ist.utl.pt}

\address{Kamila Kli\'s--Garlicka, Department of Applied Mathematics,
University of Agriculture, ul. Balicka 253c\\ 30-198 Krak\'ow, Poland.}
\email{rmklis@cyfronet.pl}

\address{
Bartosz {\L}anucha, Institute of Mathematics,
 Maria Curie-Sk{\l}odowska University, pl. M.
Curie-Sk{\l}odowskiej 1, 20-031 Lublin, Poland}
\email{bartosz.lanucha@poczta.umcs.lublin.pl}

\address{Marek Ptak, Department of Applied Mathematics,
University of Agriculture, ul. Balicka 253c\\ 30-198 Krak\'ow, Poland.}\email{rmptak@cyf-kr.edu.pl}

\thanks{The work of the first author
was partially supported by FCT/Portugal through UID/MAT/04459/2020. The research of the second and the fourth authors was financed by the Ministry of Science and Higher Education of the Republic of Poland.}

\date{\today}
\begin{document}

\begin{abstract}
The property of being shift invariant and being reflexive or transitive in the case of the space of (asymmetric) truncated Toeplitz operators,
and the space of (asymmetric) dual truncated operators is investigated. Most of the results obtained are new even for the symmetric case. A characterization of asymmetric dual truncated Toeplitz operators is also given.
\end{abstract}
\maketitle

\section{Introduction}
The investigation of {\tto}s (TTO) was inspired by \cite{sar1}. The importance of this class of operators in both symmetric and asymmetric cases was shown by many important papers (for instance \cite{sar1,GMR,BBK,bess,BCKP,BM}). The history of the dual case is shorter, however there are many results obtained in this direction (for example \cite{ding, hu,DQS,DQS2,CKLPcom,CKLPd,CKLPinv}). For arguments for the importance of (asymmetric) {\dtto}s and their connection to physics one can search in \cite{CKLPcom}.

The aim of this paper is to look more universally at restrictions of multiplication operators on the classical $L^2=L^2(\mathbb{T},m)$ space, with  $\mathbb{T}$ the unit circle and $m$ the normalized Lebesgue measure. Being inspired by the well known classical result of Brown and Halmos \cite{BH} which characterized  Toeplitz operators on the Hardy space $H^2\subset L^2$, and Sarason's definition of shift invariance \cite{sar1}  used to characterize truncated Toeplitz operators, we extend (Definition \ref{dfshinv})  Sarasons's definition  and prove the property of being shift invariant  for other classes of restrictions of the multiplication operators. 
As a consequence  we obtain, for (asymmetric)  truncated Toeplitz operators and (asymmetric) dual truncated Toeplitz operators, properties associated with the invariant subspace problem (namely, transitivity and $2$--reflexivity).


Sarason \cite{sar4} introduced a definition of reflexivity for a single operator and for operator algebras. Roughly speaking, such an algebra or  operator (i.e., the algebra generated by this operator and the identity)  have such a rich set of invariant subspaces that it determines the algebra itself. In contrast, transitivity of an algebra means that it has no common nontrivial invariant subspaces. In the language of a preanihilator (thus in the language which can concern also subspaces of operators) reflexivity means that there is a  huge set of rank--one operators in the preanihilator (so huge that this set generates the whole preanihilator). On the contrary, transitivity means that there are no rank--one operators in the preanihilator. The space of all classical Toeplitz operators on $H^2$ turned out to be transitive, see \cite{AP}, but it has a very rich set of rank--two operators in the preanihilator. Following
\cite{Larson} the word "rich"  means that rank--two operators are linearly dense in the preanihilator and then the subspace is called $2$--reflexive.

Most of the results in the paper are new even for the symmetric case, but we decided to present them as corollaries of results for the asymmetric case.
The structure of the paper is the following. For the definitions of properties mentioned above and classes of operators see Section 2.
In Section 3 (Theorem \ref{th8.12}) we give a characterization of asymmetric dual truncated Toeplitz operators (ADTTO) extending the  characterization given by the same authors in \cite{CKLPcom} for the symmetric case. This result is a crucial tool used in Section 5.
Having the general definition of shift invariance (Definition \ref{dfshinv}) we will show in Section 4 that (asymmetric) {\dtto}s are shift invariant - Corollary \ref{dshinv}. In contrast to the class of asymmetric truncated Toeplitz operators which can be characterized by shift invariance, Theorem \ref{sh-inv} (see \cite{BM}), the set of all shift invariant operators between orthogonal complements of two (possibly different) model spaces is larger, and Theorem \ref{shinth} gives a description of this new class, which also appears in the context of intertwining property (see \cite{CKLPinv}).  Using shift invariance we will show that  the space of all (asymmetric) {\tto}s is also $2$--reflexive. Next we will prove that the space of (asymmetric) {\dtto}s, even though it is not characterized by shift invariance, is $2$--reflexive. The new approach is to translate the conditions given in  Theorem \ref{th8.12} to the language of rank-two operators.

 Unfortunately the property of $2$--reflexivity is not  hereditary for subspaces. However, using the functional calculus  for (asymmetric) {\dtto}s (Theorem \ref{funcaldual}) and Bourgain's theorem we will show that the space of (asymmetric) {\dtto}s has property $\mathbb{A}_{\frac 12}(1)$ (Definition \ref{propA}) which gives $2$--reflexivity of each subspace of  (asymmetric) {\dtto}s (for example the space of analytic dual truncated Toeplitz operators).

\section{Notations and definitions}
Let $\h$, $\ka$ denote complex separable Hilbert spaces, and let $\bhk$ be the Banach space of all bounded linear operators from $\h$ to $\ka$.
A {\it rank--one} operator from $\ka$ to $\h$ will be usually denoted by $f\otimes g$, where $f\in \h$, $g\in \ka$, and it acts as $(f\otimes g)x=\langle x,g\rangle f$ for $x\in \ka$.
The weak* topology (ultraweak topology) in $\mathcal{B}(\h,\ka)$ is given by {\it trace class} operators of the form $t=\sum_{n=0}^\infty f_n\otimes g_n$ with $f_n\in\h$, $g_n\in\ka$ such that $\sum_{n=0}^\infty \|f_n\|^2< \infty$, $\sum_{n=0}^\infty \|g_n\|^2< \infty$.   Denote by $\mathcal{B}_1(\ka,\h)$ the space of all such trace class operators and by  $\|\cdot\|_1$  the trace norm. Denote also  by $\mathcal{F}_k$ the set of all operators in $\mathcal{B}_1(\ka,\h)$ of rank at most $k$.
Then $\mathcal{B}(\h,\ka)$  is a dual space to $\mathcal{B}_1(\ka,\h)$
 (see \cite[Chapter 16]{MV} for details) and the dual action is given by \[\bhk\ \times \ \mathcal{B}_1(\ka,\h)\ni(T,t)\ \mapsto\  <T,t>
=\sum_{n=0}^\infty \langle Tf_n,g_n\rangle.\]
 Let $\mathcal{S}\subset\bhk$. Then the {\it preannihilator} of $\mathcal{S}$ is given by
$$\mathcal{S}_\bot=\{t\in\mathcal{B}_1(\ka,\h)\colon <T,t>=0 \text{ for all } T\in \mathcal{S} \}.$$
Let $\mathcal{M}\subset\mathcal{B}_1(\ka,\h)$. Then the {\it annihilator} of $\mathcal{M}$ is given by
$$\mathcal{M}^\bot=\{T\in\bhk\colon  <T,t>=0 \text{ for all } t\in \mathcal{M} \}.$$
Obviously $\mathcal{S}\subset\bhk$ is weak*--closed if and only if $\mathcal{S} = (\mathcal{S}_\perp)^\perp$.

Now following \cite{Larson} we will give the reflexivity--transitivity definitions equivalent to the classical ones given in \cite{sar4, RR, LS}.
\begin{definition}\label{dfref}
Let $\mathcal{S}\subset\bhk$ be a weak${}^*$--closed subspace. Then
\begin{enumerate}\item $\mathcal{S}$ is called {\it transitive}, if
$\mathcal{S}_\bot\cap \mathcal{F}_1=\{0\}$,
\item $\mathcal{S}$ is called {\it reflexive}, if
$\mathcal{S}=\big(\mathcal{S}_\bot\cap \mathcal{F}_1\big)^\perp$,
\item $\mathcal{S}$ is called {\it $k$--reflexive}, if
$\mathcal{S}=\big(\mathcal{S}_\bot\cap \mathcal{F}_k\big)^\perp$ for $k=2,3,4\dots$.
\end{enumerate}
\end{definition}

Let $L^2$ be the space of  measurable  and square integrable functions with respect to the normalized Lebesgue measure $m$ on the unit circle $\mathbb{T}$. Let $H^2$ be  the Hardy space, that is, the space of those functions from $L^2$, which can be extended to functions  holomorphic on the whole unit disk $\mathbb{D}$. By $P^+$ we will denote the natural orthogonal projection from $L^2$ onto $H^2$ and let $P^-=I_{L^2}-P^+$. For any $\varphi\in L^\infty$ define the {\it Hankel operator} with the symbol $\varphi$ as $H_\varphi f=P^-(\varphi f)$ and the {\it Toeplitz operator} with the symbol $\varphi$ as $T_\varphi f=P^+(\varphi f)$ for all $f\in H^2$. Recall that the unilateral shift  $S$ and the Toeplitz operator $T_z$ are unitarily equivalent.
\begin{definition}\label{dfshinv}
Let $\mathcal{K}_1,\mathcal{K}_2$ be subspaces of $L^2$. An operator   $A\in \mathcal{B}(\mathcal{K}_1,\mathcal{K}_2)$ will be called {\it shift invariant} if, for any $f\in\mathcal{K}_1$ and $g\in\mathcal{K}_2$ such that $zf\in\mathcal{K}_1$ and $zg\in\mathcal{K}_2$, the following holds
\begin{equation}\label{shinvt}
  \langle A(zf),zg\rangle=\langle Af,g\rangle.
\end{equation} \end{definition}

 Let $\theta$ be a nonconstant inner function, and denote the subspace $K_\theta:=H^2\ominus \theta H^2$, which is called a {\it model space} (for a constant inner function $\theta$ we will use the convention $\kdt=\{0\}$).   By Beurling's theorem model spaces are exactly those spaces which are invariant for $S^*=T_{\bar z}$. In what follows $P_{\theta}$ will denote the orthogonal projection from $L^2$ onto
 $K_{\theta}$. Recall also that $L^2=K_{\theta}\oplus(K_{\theta})^{\perp}=K_{\theta}\oplus\theta H^2\oplus H^2_-$ where $H_-^2=\overline{zH^2}$ and let $P_{{\theta}}^\perp$, $ P^-$, $P_{\theta H^2}$ be projections onto
 $K_{\theta}^{\perp}$, $H^2_-$ and $\theta H^2$, respectively.

For  $\varphi\in L^2$  let $M_\varphi :D(M_\varphi)\to L^2$ be the densely defined multiplication operator $M_\varphi f=\varphi f$, where $D(M_\varphi)=\{f\in L^2\colon \varphi f\in L^2\}$. Note that $L^\infty\subset D(M_\varphi)$ for all $\varphi\in L^2$.
Recall from \cite{GMR} that $K_\theta^\infty:=K_\theta\cap L^\infty $ is a dense subset of $K_\theta$. 
The space $K_\theta^\perp\cap L^\infty $ is also a dense subset of $K_\theta^\perp$ as it was observed in \cite{CKLPcom}.
 For nonconstant inner functions $\theta, \alpha$ and for $\varphi\in L^2$ define
$$A_{\varphi}^{\theta,\alpha}=P_{\alpha}M_{\varphi|K_{\theta}\cap L^{\infty}},\ 
\ \text {and}\  D_{\varphi}^{\theta,\alpha}=P_{{\alpha}}^\perp M_{\varphi|K_{\theta}^{\perp}\cap L^{\infty}}.$$
If $A_{\varphi}^{\theta,\alpha}$ extends to the whole $K_\theta$ as a bounded operator, it is called  an {\it asymmetric truncated Toeplitz operator} (ATTO, see \cite{part, part2, BCKP, blicharz1}). Similarly, 
if $D_{\varphi}^{\theta,\alpha}$ extends to the whole $K_\theta^\perp$ as a bounded operator, it is called  an {\it asymmetric dual truncated Toeplitz operator} (ADTTO, see \cite{CKLPcom, CKLPd}). It is easy to verify that $(A_{\varphi}^{\theta,\alpha})^*=A_{\bar\varphi}^{\alpha,\theta}$ and $(D_{\varphi}^{\theta,\alpha})^*=D_{\bar\varphi}^{\alpha,\theta}$. 
Let us fix the notation
\begin{align*}
  \mathcal{T}(K_\theta,K_\alpha)&=\{A_{\varphi}^{\theta,\alpha}\colon\, \varphi\in
  L^2\ \mathrm{and}\ A_{\varphi}^{\theta,\alpha}\
  \mathrm{is\ bounded}\},\\
   \mathcal{T}(K^\perp_\theta, K^\perp_\alpha)&=\{D_{\varphi}^{\theta,\alpha}\colon\, \varphi\in
  L^2\ \mathrm{and}\ D_{\varphi}^{\theta,\alpha}\
  \mathrm{is\ bounded}\}
  .\end{align*}
  Recall that $D_{\varphi}^{\theta,\alpha}$ is bounded if and only if $\varphi\in L^\infty$ (see \cite{ding,CKLPd}).
In the symmetric case, i.e., $\theta=\alpha$, we will use a shorter notation:
$A_{\varphi}^{\theta}$, $D_{\varphi}^{\theta}$ and $\mathcal{T}(K_\theta)$ and $\mathcal{T}(K^\perp_\theta)$, respectively. We will also use names of such operators without the word "asymmetric" and use the abbreviations TTO and DTTO, respectively.

There is a natural conjugation (an antilinear isometric involution) on $L^2$ connected with each model space (see for instance \cite{GP,CKLPcon2, CKLPconj1}).
For an inner function $\theta$ define
\begin{equation}\label{ctheta} C_{\theta}f=\theta\bar z \bar f\quad \text{for}\quad f\in L^2.\end{equation}
Then $\langle C_\theta f, C_\theta g\rangle=\langle  g, f\rangle$ for $f,g\in L^2$. One can easily verify that
\begin{equation}\label{31}
C_{\theta}M_{\varphi}C_{\theta}=M_{\bar{\varphi}}.
\end{equation}

It is well known (\cite{GP, CKLPconj1}) that $C_{\theta}$ preserves $K_{\theta}$.  Moreover, $C_{\theta}(\theta H^2)=H^2_-$ and $C_{\theta}(H^2_-)=\theta H^2$, so $C_{\theta}$ also preserves $(K_{\theta})^{\perp}$.

\section{A characterization of $D^{\theta,\alpha}_\varphi$}

In this section our goal is to give a characterization of asymmetric dual truncated Toeplitz operators (Theorem \ref{th8.12}). A general idea of this characterization and its proof is similar to \cite[Theorem 27]{CKLPcom}. As in \cite{CKLPcom}, we will first consider compressions of ADTTO's to certain subspaces of $K_\theta^\perp$ and $K_\alpha^\perp$.

 Let $\theta,\alpha$ be two inner functions. Using the decompositions  $K_\theta^\perp=\theta H^2\oplus H^2_- $ and $K_\alpha^\perp=\alpha H^2\oplus H^2_- $ one can write each operator $D\in \mathcal{B}(K_\theta^\perp, K_\alpha^\perp)$ as a matrix
$$D=\begin{bmatrix}\dtha&\dba\\ \dc &\dd\end{bmatrix}.$$
In particular, for $\varphi\in L^{\infty}$, we obtain
$$D_\varphi^{\theta,\alpha}=\begin{bmatrix}\hat{T}_\varphi^{\theta,\alpha}&\check{\Gamma}_{\varphi}^\alpha\\ \hat{\Gamma}_\varphi^\theta &\check{T}_\varphi\end{bmatrix}=\begin{bmatrix}\hat{T}_\varphi^{\theta,\alpha}&(\hat{\Gamma}_{\bar{\varphi}}^{{\alpha}})^*\\
\hat{\Gamma}_\varphi^{\theta}&\check{T}_\varphi\end{bmatrix},$$
where
\begin{equation}\label{hat} \hat{T}_\varphi^{\theta,\alpha}=P_{\alpha H^2}M_{\varphi|\theta H^2}, \quad \hat{\Gamma}_\varphi^{\theta}=P^-M_{\varphi|\theta H^2}
\end{equation}
and
\begin{equation}\label{czek}
 \check{\Gamma}_\varphi^{\alpha}=P_{\alpha H^2}M_{\varphi| H^2_-}, \quad  \check{T}_\varphi=P^-M_{\varphi| H^2_-}.
\end{equation}
Let us denote
\begin{align*}
	\mathcal{T}(\theta H^2,\alpha H^2)&
	=\{\hat{T}\in \mathcal{B}(\theta H^2,\alpha H^2):\   \hat{T}=\hat{T}_\varphi^{\theta,\alpha}\ \text{for some }\varphi\in L^\infty\},\\
\mathcal{T}( H^2_-)&=\{\check{T}\in \mathcal{B}( H^2_-):\  \check{T}=\check{T}_\varphi\ \text{for some }\varphi\in L^\infty\},\\
\mathcal{T}(\theta H^2,H^2_-)&=\{\hat{\Gamma}\in \mathcal{B}(\theta H^2,H^2_-): \ \hat{\Gamma}=\hat{\Gamma}_{\varphi}^{\theta}
=P^-M_{\varphi|{\theta H^2}}\ \text{for }\varphi\in L^\infty\},\\
\mathcal{T}(H^2_-,\alpha H^2)&=\{\check{\Gamma}\in \mathcal{B}({H^2_-},\alpha H^2): \  \check{\Gamma}=\check{\Gamma}_{\varphi}^{\alpha}=P_{\alpha H^2}M_{\varphi|{H^2_-}}\ \text{for }\varphi\in L^\infty\}.
\end{align*}
As in \cite{CKLPcom} we will write $\mathcal{T}(\theta H^2)$ and $\hat{T}_\varphi^{\theta}$ instead of  $\mathcal{T}(\theta H^2,\theta H^2)$ and $\hat{T}_\varphi^{\theta,\theta}$, respectively.

Each of the operators in \eqref{hat} and \eqref{czek} can be similarly defined for arbitrary $\varphi\in L^2$. In that case $\hat{\Gamma}_\varphi^{\theta}$ and $\hat{T}^{\theta,\alpha}_\varphi$ are defined on a dense subset $\theta H^\infty$ of $\theta H^2$, while $\check{\Gamma}_\varphi^{\alpha}$ and $\check{T}_\varphi$ are defined on $H_-^\infty=H^2_-\cap L^\infty=\overline{zH^\infty}$ which is a dense subset of $H^2_-$. However, in a moment we will justify the fact that, in a sense, one needs only to consider symbols from $L^\infty$.


Recall firstly that the operator $J$, defined by $J \colon L^2\to L^2$, $Jf(z)=\bar z \overline{f(z)}$, $z\in \mathbb{T}$, is an antilinear involution. Moreover, $J^{-1}=J=J^\sharp$ (by $\sharp$ we denote the antilinear adjoint) and $J(H^2)=H_-^2$, $J(H_-^2)=H^2$. More properties of antilinear operators can be found for example in \cite{PSW}.
Recall also that the multiplication operator $M_\theta$ maps $H^2$ bijectively onto $\theta H^2$ and $M^{-1}_\theta=M_{\bar\theta}$. Moreover, each of the operators $J$, $M_\theta$ and $M_{\bar{\theta}}$ preserves $L^\infty$.

For $\varphi\in L^2$ denote by $T_\varphi$ and $H_\varphi$ the classical Toeplitz and Hankel operators with the symbol $\varphi$, respectively (both densely defined on $H^\infty$).

The following proposition is not difficult to verify (compare with \cite[Proposition 20]{CKLPcom}).

\begin{proposition}\label{propc2} Let $\theta$ and $\alpha$ be two nonconstant inner functions and let $\varphi\in L^2$. Then
\begin{enumerate}
  \item $\hat{T}^{\theta,\alpha}_{\varphi|\theta H^\infty}=P_{\alpha H^2}M_{\varphi|\theta H^\infty}=M_\alpha T_{\bar\alpha\varphi\theta} M_{\bar\theta|\theta H^\infty}=M_\alpha T_{\bar\alpha\varphi |\theta H^\infty}$;
   \item $\hat{\Gamma}^\theta_{\varphi|\theta H^\infty}=P^-M_{\varphi|\theta H^\infty}=H_{\varphi\theta}M_{\bar\theta|\theta H^\infty}$;
   \item $\check{T}_{\varphi|H_-^\infty}=P^-M_{\varphi|H^\infty_-}=JT_{\bar\varphi}J_{|H_-^\infty}$;
   \item $\check{\Gamma}^\alpha_{\varphi|H^\infty_-}=P_{\alpha H^2}M_{\varphi| H^\infty_-}=M_\alpha H^*_{\alpha\bar\varphi|H^\infty_-}$.
	\end{enumerate}
\end{proposition}

It now follows from Proposition \ref{propc2}(1) that $\hat{T}\in \mathcal{T}(\theta H^2,\alpha H^2)$ if and only if $ M_{\bar\alpha} \hat{T} M_{\theta|H^2}\in \mathcal{T}(H^2)$ and so $\mathcal{T}({\theta H^2},\alpha H^2)$ is isomorphic to the space of classical Toeplitz operators $\mathcal{T}({H^2})$. Moreover, each  $\hat{T}^{\theta,\alpha}_\varphi$ is uniquely determined by its symbol and extends to a bounded operator on $\theta H^2$ if and only if $\varphi\in L^\infty$ (since classical Toeplitz operators have such properties). Similarly, $\mathcal{T}({H^2_-})$ is isomorphic to $\mathcal{T}({H^2})$, each  $\check{T}_\varphi$ is uniquely determined by its symbol and extends to a bounded operator on $H^2_-$ if and only if $\varphi\in L^\infty$.

On the other hand, both $\mathcal{T}(\theta H^2,H^2_-)$ and $\mathcal{T}(H^2_-,\alpha H^2)$ are isomorphic to the space of all Hankel operators $\mathcal{H}(H^2,H^2_-)$. It follows from the properties of classical Hankel operators that $\hat{\Gamma}_{\varphi}^{\theta}$ and $\check{\Gamma}_{\varphi}^{\alpha}$ may be bounded even for $\varphi\notin L^\infty$, and
\begin{equation}
\label{zerro}
\hat{\Gamma}_{\varphi}^{\theta}=\hat{\Gamma}_{\psi}^{\theta}\quad\text{if and only if}\quad (\varphi-\psi)\perp \overline{\theta z H^2}
\end{equation}
and
\begin{equation*}
\check{\Gamma}_{\varphi}^{\alpha}=\check{\Gamma}_{\psi}^{\alpha}\quad\text{if and only if}\quad (\varphi-\psi)\perp {\alpha z H^2}.
\end{equation*}
In particular, $\hat{\Gamma}^\theta_\varphi=0$ if $\varphi\in\bar\theta H^\infty\supset H^\infty$ and
$\check{\Gamma}^\alpha_\varphi=0$ if $\varphi\in\alpha\overline{ H^\infty}\supset \overline{H^\infty}$. However, since each bounded Hankel operator has a symbol from $L^\infty$ \cite[Theorem 1.3, Chapter 1]{VVP}, we see that the same is true for operators from $\mathcal{T}(\theta H^2,H^2_-)$ or $\mathcal{T}(H^2_-,\alpha H^2)$. Thus operators with bounded symbols form the spaces $\mathcal{T}(\theta H^2,H^2_-)$ and $\mathcal{T}(H^2_-,\alpha H^2)$.

Observe also that 
for $\varphi_1,\varphi_2\in L^{\infty}$,
 $$\hat{T}^\theta_{\varphi_1}\hat{T}^\theta_{\varphi_2}=M_\theta T_{\varphi_1} T_{\varphi_2} M_{\bar\theta|\theta H^2}\quad\text{and}\quad\check{T}_{\varphi_1}\check{T}_{\varphi_2}=JT_{\bar\varphi_1} T_{\bar\varphi_2}  J.$$
 It thus follows from the properties of classical Toeplitz operators that if one of the functions $\varphi_1$, $\varphi_2$ belongs to $H^{\infty}$, then
\begin{equation}\label{plus}
 \hat{T}^\theta_{\bar\varphi_1}\hat{T}^\theta_{\varphi_2}=\hat{T}^\theta_{\bar\varphi_1\varphi_2}\quad\text{and}\quad\check{T}_{\varphi_1}\check{T}_{\bar\varphi_2}=\check{T}_{\varphi_1\bar\varphi_2}.
 \end{equation}

It is well known that classical Toeplitz and Hankel operators can be characterized in terms of compressions of $M_z$ to $H^2$ and $H^2_-$ (for more details see  \cite[Theorem 4.16]{GMR} and \cite[Theorem 1.8, Chapter 1]{VVP}). As a consequence of Proposition \ref{propc2} we get the following.

\begin{theorem}\label{thm5.7}Let $\theta$ and $\alpha$ be two nonconstant inner functions.
	\begin{enumerate}
		\item Let $\hat{T}\in \mathcal{B}(\theta H^2,\alpha H^2)$. Then $\hat{T}\in \mathcal{T}(\theta H^2,\alpha H^2)$ if and only if $ \hat{T}=\hat{T}^\alpha_{\bar z}\hat{T}\hat{T}^{\theta}_z$, and in that case $\hat{T}=\hat{T}^{\theta,\alpha}_\varphi$ with $\varphi=\bar\theta\hat{T}(\theta)+\alpha \overline{\hat{T}^*(\alpha)}-\alpha\overline{\langle \hat{T}^*(\alpha),\theta\rangle\theta}\in L^\infty$.
		\item Let $\check T\in \mathcal{B}(H^2_-)$. Then $\check T \in \mathcal{T}(H^2_-)$ if and only if $\check T=\check T_{ z}\check T \check T_{\bar z}$, and in that case $\check T=\check T_\varphi$ with $\varphi=z\check T \bar z+\bar z \overline{\check T^*\bar z}-\langle \check T\bar z, \bar z\rangle\in L^\infty$.
		\item
		Let $\hat\Gamma\in \mathcal{B}(\theta H^2,H^2_-)$. Then $\hat\Gamma\in \mathcal{T}(\theta H^2,H^2_-)$ if and only if $ \check T_z \hat\Gamma=\hat\Gamma \hat T^\theta_z$, and in that case $\hat\Gamma=\hat\Gamma^\theta_\varphi$ with $P^-(\theta\varphi)=\hat\Gamma\theta$. Moreover, there exists such $\varphi$, which belongs to $L^\infty$.
		\item Let  $\check\Gamma \in \mathcal{B}(H^2_-,\alpha H^2)$. Then $\check\Gamma \in \mathcal{T}(H^2_-,\alpha H^2)$ if and only if $ \check\Gamma \check T_{\bar z}=\hat T_{\bar z}^{\alpha}\check\Gamma$, and in that case $\check\Gamma=\check\Gamma^\alpha_\varphi$ with $P^-(\alpha\bar\varphi)=\check\Gamma^*\alpha$. Moreover, there exists such a $\varphi$, which belongs to $L^\infty$.
	\end{enumerate}
\end{theorem}

\begin{proof}
	For the proof see \cite{CKLPcom}.
\end{proof}

We will now consider operators of the form
	\begin{equation*}
D=\begin{bmatrix}\hat{T}_{\varphi_1}^{\theta,\alpha}&\check{\Gamma}_{\varphi_2}^{\alpha}
\\\hat{\Gamma}_{\varphi_3}^{\theta}&\check{T}_{\varphi_4}\end{bmatrix}
\end{equation*}
with $\varphi_i\in L^2$ for $i=1,2,3,4$. Note that if $D$ given above is bounded, then $\hat{T}_{\varphi_1}^{\theta,\alpha}$ and $\check{T}_{\varphi_4}$ are also bounded and so, as mentioned above, necessarily $\varphi_1,\varphi_4\in L^\infty$. On the other hand, even though for bounded $D$ the compressions $\check{\Gamma}_{\varphi_2}^{\alpha}$ and $\hat{\Gamma}_{\varphi_3}^{\theta}$ are also bounded, the functions $\varphi_2$ and $\varphi_3$ may not belong to $L^\infty$ (but there exist $\psi_2,\psi_3\in L^\infty$ such that $\check{\Gamma}_{\varphi_2}^{\alpha}=\check{\Gamma}_{\psi_2}^{\alpha}$ and $\hat{\Gamma}_{\varphi_3}^{\theta}=\hat{\Gamma}_{\psi_3}^{\theta}$).

We will now study relations of the operators \eqref{hat} and \eqref{czek} with respect to the conjugation $C_\theta$ (see \eqref{ctheta}). Recall that $C_\theta$ can be expressed as $C_\theta= M_\theta J=J M_{\bar\theta}$, hence we obtain the following.
\begin{proposition}\label{cor5.3}
For $\varphi\in L^\infty$,
\begin{enumerate}
  \item $\hat{T}^{\theta,\alpha}_\varphi=C_\alpha \check{T}_{\alpha\bar\varphi\bar{\theta}} C_{\theta|\theta H^2}=C_\alpha \check{T}_{\alpha}\check{T}_{\bar\varphi} \check{T}_{\bar{\theta}}C_{\theta|\theta H^2}$;
  \item $\check{T}_{\varphi}=(P^-C_{\alpha}M_{\bar\theta})_{|\theta H^2}\hat{T}^{\alpha,\theta}_{\bar{\varphi}}
(M_{\theta}C_{\alpha})_{|H^2_-}$;
    \item $\hat{\Gamma}^\theta_\varphi= C_\theta \check{\Gamma}^\theta_{\bar\varphi} C_{\theta|\theta H^2}$.
\end{enumerate}
\end{proposition}

\begin{proof}
	The proof of (3) can be found in \cite{CKLPcom}. A slight modification of the proof of \cite[Proposition 22]{CKLPcom} (for $\alpha=\theta$) gives
\begin{displaymath}
\begin{split}
	\hat{T}^{\theta,\alpha}_\varphi&=M_\alpha T_{\bar{\alpha}\varphi\theta} M_{\bar\theta|\theta H^2}=M_\alpha (JP^-J)M_{\bar{\alpha}\varphi\theta} J (JM_{\bar\theta})_{|\theta H^2}\\
	&=(M_\alpha J) P^-M_{{\alpha}\bar\varphi\bar\theta}P^-JM_{\bar\theta|\theta H^2}=C_\alpha\check{T}_{{\alpha}\bar\varphi\bar\theta} C_{\theta|\theta H^2}
	=C_\alpha \check{T}_{\alpha}\check{T}_{\bar\varphi} \check{T}_{\bar{\theta}}C_{\theta|\theta H^2},
\end{split}
\end{displaymath}
where the last equality follows from \eqref{plus}. Hence (1) holds.

To prove (2) recall that $JM_{\bar\theta}=M_{\theta}J$ and $JP^+=P^-J$, which implies that
$$JT_{\bar\theta}=P^-M_{\theta}J_{|H^2}\quad\text{and}\quad T_{\alpha}J_{|H^2_-}=JP^-M_{\bar{\alpha}|H^2_-}.$$
Since $J=C_{\alpha}M_{\alpha}=M_{\bar\theta}C_{\theta}$, we get by Proposition \ref{propc2}:
\begin{align*}
\check{T}_{\varphi}&=J{T}_{\bar\varphi}J_{|H^2_-}= J T_{\bar \theta} T_{\bar\alpha\bar\varphi \theta}T_{\alpha}J_{|H^2_-}=P^- M_{\theta}JT_{\bar\alpha\bar\varphi \theta}JP^-M_{\bar\alpha|H^2_-}\\
&=P^- M_{\theta}C_{\alpha}M_{\alpha}T_{\bar\alpha\bar\varphi \theta}M_{\bar\theta}C_{\theta}P^-M_{\bar\alpha|H^2_-}=P^- M_{\theta}C_{\alpha}\hat{T}^{\theta,\alpha}_{\bar\varphi}C_{\theta}M_{\bar\alpha|H^2_-}.
\end{align*}
The result follows, since
$$M_{\theta}C_{\alpha}=
M_{\alpha}C_{\theta}=C_{\theta}M_{\bar\alpha}=C_{\alpha}M_{\bar\theta}.$$
\end{proof}

Note that for arbitrary $\varphi\in L^2$ the equalities (1)--(3) in Proposition \ref{cor5.3} hold on the set of bounded functions.

Finally, we are ready to give the following characterization of ADTTO's.

\begin{theorem}\label{th8.12}
	Let $\theta$ and $\alpha$ be inner functions and let $D\in  \mathcal{B}(K_\theta^\perp,K_\alpha^\perp)$.
 Then the operator $D$ is an asymmetric dual truncated Toeplitz operator,
	$D\in\mathcal{T}(K_\theta^\perp,K_\alpha^\perp)$, if and only if the following conditions hold:
	\begin{enumerate}
		\item $\dtha=\hat{T}^\alpha_{\bar{z}}\dtha\hat{T}^\theta_z$;
		\item $\dd=(P^-C_{\alpha}M_{\bar\theta})_{|\theta H^2}\left(\dtha\right)^*
			(M_{\theta}C_{\alpha})_{|H^2_-}$;
		\item $\dc\hat{T}_z^\theta=\check{T}_z\dc$ and $(\dba)^*\hat{T}_z^\alpha =\check{T}_z (\dba)^*$;
		\item $P^-(D(\theta))=P^-(\theta\alpha\overline{D^*(\alpha)})$ and $P^-(D^*(\alpha))=P^-(\theta\alpha\overline{D(\theta)})$.
				\end{enumerate}
In that case, $D=D_{\varphi}^{\theta,\alpha}$ with $\varphi\in L^{\infty}$ given by
\begin{equation}\label{sym}
\varphi\ =\ \bar{\theta}P_{\alpha H^2}(D(\theta))+\alpha\overline{P_{\theta H^2}(D^*(\alpha))}-\alpha\overline{\langle D^*(\alpha),\theta\rangle\theta}.
\end{equation}
\end{theorem}

\begin{proof}
The proof is similar to the proof of \cite[Theorem 7.1]{CKLPcom}.
Assume firstly that $D=D_{\varphi}^{\theta,\alpha}$ with $\varphi\in L^{\infty}$. Then
$$D_\varphi^{\theta,\alpha}=\begin{bmatrix}\hat{T}^{\theta,\alpha}_\varphi &\check{\Gamma}_{\varphi}^\alpha\\\hat{\Gamma}_\varphi^\theta&\check{T}_\varphi\end{bmatrix}=
\begin{bmatrix}\hat{T}^{\theta,\alpha}_\varphi&(\hat{\Gamma}^\alpha_{\bar{\varphi}})^*\\\hat{\Gamma}^\theta_\varphi&\check{T}_\varphi\end{bmatrix}.$$
Note that (1) follows from Theorem \ref{thm5.7}(1). Moreover, (2) is satisfied by Proposition \ref{cor5.3}(2) and (3) is satisfied by Theorem \ref{thm5.7}. Similarly as in the proof of  \cite[Theorem 7.1]{CKLPcom}, it is straightforward to verify that $D=D_{\varphi}^{\theta,\alpha}$ satisfies (4).

Assume now that $D=\begin{bmatrix}\dtha&\dba\\\dc&\dd\end{bmatrix}\in  \mathcal{B}(K_\theta^\perp,K_\alpha^\perp)$ satisfies (1)--(4). It then follows from condition (1) and Theorem \ref{thm5.7}(1) that $\dtha=\hat{T}_{\varphi}^{\theta,\alpha}$ for $\varphi\in L^{\infty}$ given by
\begin{equation}
\label{xr}
\varphi={\bar{\theta}}\left(\dtha\right)(\theta)+{\alpha}\overline{\left(\dtha\right)^*(\alpha)}-\alpha\overline{\langle \left(\dtha\right)^*(\alpha),\theta\rangle\theta}.
\end{equation}
By (2) and Proposition \ref{cor5.3}(2),
\begin{align*}
\dd&=(P^-C_{\alpha}M_{\bar\theta})_{|\theta H^2}\left(\dtha\right)^*
(M_{\theta}C_{\alpha})_{|H^2_-}\\
&=(P^-C_{\alpha}M_{\bar\theta})_{|\theta H^2}\hat{T}^{\alpha,\theta}_{\bar{\varphi}}
(M_{\theta}C_{\alpha})_{|H^2_-}=\check{T}_{{\varphi}}.
\end{align*}
Condition (3) and Theorem \ref{thm5.7}(3)--(4) imply that $\dc=\hat{\Gamma}_{\psi_1}^{\theta}$ and $\dba=\check{\Gamma}_{\psi_2}^{\alpha}$ for some $\psi_1,\psi_2\in L^\infty$. 
Using condition (4) one can show that $\hat{\Gamma}_{\psi_1}^{\theta}=\hat{\Gamma}_{\varphi}^{\theta}$ and $\check{\Gamma}_{\psi_2}^{\alpha}=\check{\Gamma}_{\varphi}^{\alpha}$, hence $D=D_{\varphi}^{\theta,\alpha}$ (see
the proof of \cite[Theorem 7.1]{CKLPcom} for details).
Moreover, \eqref{sym} follows from \eqref{xr}.

\end{proof}

\begin{remark}  By \eqref{sym}, the symbol $\varphi\in L^\infty$ of an asymmetric dual truncated Toeplitz operator $D$ can be obtained by calculating $D(\theta)$ and $D^*(\alpha)$. As in \cite[Remark 28]{CKLPcom}, $\varphi$ can also be calculated using $D(\bar z)$ and $D^*(\bar z)$. To see this let $\varphi=\varphi^-+\varphi^+$, $\varphi^-\in H^2_-$, $\varphi^+\in H^2$ and let $\hat\varphi(0)$ denote the $0$--th Fourier coefficient  of $\varphi$. Then, by the fact that $C_{\alpha}D^{\theta,\alpha}_{\varphi} C_{\theta|K_{\theta}^\perp}=D^{\theta,\alpha}_{\alpha\bar\varphi\bar{\theta}}$, we have
	\begin{enumerate}
		\item $\hat\varphi (0)=\langle 1, \bar\varphi\rangle=\langle \alpha,  D^{\theta,\alpha}_{\alpha\bar\varphi\bar{\theta}}\theta\rangle=\langle \alpha,  C_{\alpha}D^{\theta,\alpha}_{\varphi} C_{\theta}\theta\rangle=\langle D \bar z,\bar z\rangle$,
		\item $\varphi^+= P^+\bar\theta D^{\alpha,\theta}_{\bar\alpha\varphi\theta}\alpha=P^+\bar\theta C_{\theta}D^{\alpha,\theta}_{\bar\varphi}C_{\alpha}\alpha=J P^-(D^*\bar z)$,
		\item $
		\varphi^- =\overline{P^+ ({\bar\alpha}  D^{\theta,\alpha}_{\alpha\bar\varphi\bar\theta}\theta) }-\hat\varphi(0)=\overline{P^+ ({\bar\alpha}  C_{\alpha}D^{\theta,\alpha}_{\varphi}C_{\theta}\theta) }-\hat\varphi(0)\\\phantom{ \varphi^-}=\overline{JP^- ( D(\bar z))}-\hat\varphi(0)  = P^-(zD (\bar z)).
		$
	\end{enumerate}
Hence
\begin{equation}\label{sym2}
\varphi=P^-(zD (\bar z))+J P^-(D^*\bar z)\in H^2_-+H^2.
\end{equation}
Note that the decomposition \eqref{sym2} is orthogonal while \eqref{sym} in general is not.\end{remark}

\vspace{0.2cm}

Let $\alpha, \theta$ be nonconstant inner functions and denote by $\mathcal{A}(K_\theta^\perp, K_\alpha^\perp)$ the set of all asymmetric dual truncated Toeplitz operators with an analytic symbol, i.e.,
$$\mathcal{A}(K_\theta^\perp, K_\alpha^\perp)=\{D^{\theta,\alpha}_\varphi\colon \varphi\in H^\infty \}.$$
\begin{remark}\label{danal}
 Let $D\in \mathcal{B}(K_\theta^\perp, K_\alpha^\perp)$. Then $D\in \mathcal{A}(K_\theta^\perp, K_\alpha^\perp)$ if and only if $D$ satisfies the conditions of Theorem \ref{th8.12} and $P^- (zD(\bar z))=0$. The last condition means that $D(\bar z)\perp \bar z H^2_-$.\end{remark}

\section{Shift invariance of ATTO and ADTTO}
The property of being shift invariant was firstly considered in \cite{sar1} where operators on $\kdt$  were considered and
 it was shown that an operator $A\in \mathcal{B}(K_\theta)$ is a truncated Toeplitz operator if and only if $A$ is  shift invariant.
Let $\alpha$, $\theta$ be nonconstant inner functions.
Shift invariance for the asymmetric case $A\in\mathcal{B}(\kdt,\kda)$ was firstly considered in \cite{CKP}, where it was assumed that $\alpha\leqslant\theta $. Later the general case was proved in \cite{BM}. Namely:
\begin{theorem}[\cite{BM}]\label{sh-inv}
Let $\alpha, \theta$ be nonconstant inner functions and let $A\colon \kdt\to\kda$ be a bounded linear operator.  Then $A\in \mathcal{T}(K_\theta, K_\alpha)$ if and only if $A$ is shift invariant.
\end{theorem}

Now let us consider the dual case.
In the context of the definition of shift invariance (Definition \ref{dfshinv}) the following lemma will be useful.
\begin{lemma}\label{l8.1}
Let $f\in K_\theta^\perp$. Then $z f\in K_\theta^\perp$ if and only if $f$ is ortho\-go\-nal to the function $\bar z\in L^2$, i.e., $f\in \theta H^2\oplus \bar z^2\overline{H^2}$.
\end{lemma}

\begin{proof}
Let $f=\bar z \bar g+\theta h\in K_\theta^\perp$ with $g, h\in H^2$. Note that $z f=\bar g+z\theta h\in K_\theta^\perp$ only if
$\overline{g(0)}=0$. Thus
$g=zg_1$, for some $g_1\in H^2$.
\end{proof}

Let $\alpha$, $\theta$ be  nonconstant inner functions. Recall that $K_\theta^\perp=\theta H^2 \oplus H^2_-$, $K_\alpha^\perp=\alpha H^2 \oplus H^2_-$.
Let us consider the set
\begin{equation}\label{ttmat}
\mathcal{T}^2(K_\theta^\perp,K_\alpha^\perp)=\left\{\begin{bmatrix}\hat{T}_{\varphi_1}^{\theta,\alpha}&\check{\Gamma}_{\varphi_2}^{\alpha}
\\\hat{\Gamma}_{\varphi_3}^{\theta}&\check{T}_{\varphi_4}\end{bmatrix}\colon \varphi_i\in L^\infty\  \text{for} \ i=1,2,3,4 \right\}.
\end{equation}%
Now we will characterize all shift invariant operators in $\mathcal{B}(K_\theta^\perp,K_\alpha^\perp)$.

\begin{theorem}\label{shinth}
Let $\alpha, \theta$ be nonconstant inner functions. 
The operator $D\in \mathcal{B}(K_\theta^\perp,K_\alpha^\perp)$ is shift invariant if and only if $D\in \mathcal{T}^2(K_\theta^\perp,K_\alpha^\perp)$.
%
\end{theorem}

The theorem above is new even in the symmetric case, i.e., $\theta=\alpha$. An immediate consequence of it is the following.
\begin{corollary}\label{dshinv}
Let $D\in \mathcal{B}(K_\theta^\perp, K_\alpha^\perp)$. If $D$ is an asymmetric dual truncated Toeplitz operator, then $D$ is shift invariant, i.e.,
\begin{equation}\label{eq8.1}
  \langle D z f,z g\rangle=\langle D f,g\rangle
\end{equation}
for all $f,g\in K_\theta^\perp$ such that $f, g$ are orthogonal to $\bar z$.
\end{corollary}

Another consequence of Theorem \ref{shinth} is that, in contrast to truncated Toeplitz operators which can be characterized by shift invariance (see \cite{sar1}), shift invariance is not a sufficient condition to prove that an operator is a dual truncated Toeplitz operator.

\begin{proof}[Proof of Theorem \ref{shinth}]
Assume that $D$ belongs to the set given by \eqref{ttmat}. Taking into consideration Lemma \ref{l8.1} let us take
 $f=\bar z^2 \bar f_-+\theta f_+$, $g=\bar z^2 \bar g_-+\alpha g_+$, $  f_-, f_+, g_-, g_+\in H^2$. Then we have
\begin{multline*}
\langle DM_z f, M_z g\rangle=\langle D(\bar z\bar f_-+\theta z f_+),\bar z\bar g_-+\alpha zg_+\rangle\\
=\langle \hat T_{\varphi_1}^{\theta,\alpha} \theta zf_+,\alpha zg_+\rangle+ \langle \check{\Gamma}_{\varphi_2}^\alpha\bar z\bar f_-,\alpha zg_+\rangle\\+\langle \hat{\Gamma}_{\varphi_3}^\theta\theta z f_+,\bar z\bar g_-\rangle+\langle \check{T}_{\varphi_4} \bar z\bar f_-,\bar z\bar g_-\rangle.
\end{multline*}
Note firstly that, by Theorem \ref{thm5.7},
\begin{equation}\label{t1}
 \langle \hat T_{\varphi_1}^{\theta,\alpha} \theta zf_+,\alpha zg_+\rangle= \langle \hat{T}_{\bar z}^\alpha\hat T_{\varphi_1}^{\theta,\alpha} \hat{T}_z^\theta \theta f_+,\alpha g_+\rangle=\langle \hat T_{\varphi_1}^{\theta,\alpha} \theta f_+,\alpha g_+\rangle.
\end{equation} Moreover,  we have
\begin{multline}\label{t2}
  \langle \check{\Gamma}_{\varphi_2}^\alpha\bar z\bar f_-,\alpha zg_+\rangle=\langle \hat{T}_{\bar z}^\alpha\check{\Gamma}_{\varphi_2}^\alpha  \bar z\bar f_-,\alpha g_+\rangle\\=\langle \check{\Gamma}_{\varphi_2}^\alpha \check{T}_{\bar z} \bar z\bar f_-,\alpha g_+\rangle=\langle \check{\Gamma}_{\varphi_2}^\alpha \bar z^2\bar f_-,\alpha g_+\rangle.
\end{multline}
Recall that  $\check{T}_{\bar z}\check{T}_{z|\bar z {H^2_-}} =I_{|H^2_-}$, see  \cite[Corollary 6.5]{CKLPcom}.
 Hence,  we get
\begin{multline}\label{t3}
  \langle \hat{\Gamma}_{\varphi_3}^\theta\theta z f_+,\bar z\bar g_-\rangle=\langle \check{T}_{\bar z}\hat{\Gamma}_{\varphi_3}^\theta \hat{T}^\theta_z\theta  f_+,\bar z^2\bar g_-\rangle\\=
  \langle \check{T}_{\bar z}\check{T}_z\hat{\Gamma}_{\varphi_3}^\theta\theta  f_+,\bar z^2\bar g_-\rangle=\langle \hat{\Gamma}_{\varphi_3}^\theta\theta  f_+,\bar z^2\bar g_-\rangle.
\end{multline}

Finally, by Theorem \ref{thm5.7} we have
\begin{equation}\label{t4}
\langle \check{T}_{\varphi_4} \bar z\bar f_-,\bar z\bar g_-\rangle=\langle \check{T}_{z}\check{T}_{\varphi_4} \check{T}_{\bar z}\bar z\bar f_-,\bar z\bar g_-\rangle=
\langle \check{T}_{\varphi_4}\bar z^2\bar f_-,\bar z^2\bar g_-\rangle.%
\end{equation}
Taking \eqref{t1}--\eqref{t4} into account we have obtained $\langle D z f, z g\rangle=\langle D f,  g\rangle.$

For the inverse implication assume  that $D\in\mathcal{B}(K_\theta^\perp, K_\alpha^\perp)$ is shift invariant and write $D$ as the matrix $$D=\begin{bmatrix}\daa&\dba\\\dc&\dd\end{bmatrix}.$$
Since $D$ is shift invariant, for  $f=\theta f_+$, $g=\alpha g_+$, $   f_+,  g_+\in H^2$, we have
 \begin{multline*} \langle \daa  \theta f_+,\alpha g_+\rangle=\langle \daa z \theta f_+,z\alpha g_+\rangle\\=\langle \hat{T}_{\bar z}^\alpha\daa \hat{T}_z^\theta \theta f_+,\alpha g_+\rangle . \end{multline*}
By  Theorem \ref{thm5.7} there is $\varphi_1\in L^\infty$ such that $\daa=\hat{T}^{\theta,\alpha}_{\varphi_1}$. Next, for  $f=\bar z^2 \bar f_-$, $g=\bar z^2 \bar g_-$, $   f_-,  g_-\in H^2$, by \eqref{eq8.1} we have
 \[  \langle \dd  \bar z \bar f_-,\bar z \bar g_-\rangle=\langle \dd \bar z^2  f_-,\bar z^2\bar g_-\rangle=\langle \check{T}_{ z}\dd \check{T}_{\bar z} \bar z\bar f_-,\bar z \bar g_-\rangle . \]
By Theorem \ref{thm5.7} there is $\varphi_4\in L^\infty$ such that $\dd=\check{T}_{\varphi_4}$. Now, for  $f=\theta f_+$, $g=\bar z^2 \bar g_-$, $  f_+, g_-\in H^2$, by \eqref{eq8.1} we have
\begin{multline*}
\langle \dc \hat{T}_z^\theta \theta  f_+,\bar z \bar g_-\rangle=\langle \dc  z\theta  f_+,z\bar z^2\bar g_-\rangle\\
=\langle \dc  \theta  f_+,\bar z^2\bar g_-\rangle=\langle \check{T}_z\dc  \theta  f_+,\bar z\bar g_-\rangle .
\end{multline*}
By Theorem \ref{thm5.7} there is $\varphi_3\in L^\infty$ such that $\dc=\hat{\Gamma}^\theta_{\varphi_3}$.
The equality $\dba=\check{\Gamma}^\alpha_{\varphi_2}$ for some $\varphi_2\in L^\infty$ can be shown similarly.
\end{proof}

\section{Reflexivity and transitivity}

The subspace $\mathcal{T}(H^2)$ of all Toeplitz operators is transitive and $2$--reflexive \cite{AP}.  We show below that the space of (asymmetric) (dual) truncated Toeplitz operators has similar properties.
\begin{theorem}\label{refl} Let $\theta$, $\alpha$ be two nonconstant inner functions.
The space $\mathcal{T}(K_\theta, K_\alpha)$ of all bounded asymmetric truncated Toeplitz operators is weak*(WOT)--closed, transitive and $2$--reflexive.
\end{theorem}

\begin{corollary}\label{refla}
The space $\mathcal{T}(K_\theta)$ of all bounded truncated Toeplitz operators is weak*(WOT)--closed, transitive and $2$--reflexive.
\end{corollary}

\begin{proof}[Proof of Theorem \ref{refl}]
 Weak*(WOT)--closedness of  $\mathcal{T}(K_\theta, K_\alpha)$ was shown in \cite[Theorem 4.2]{sar1} for $\alpha=\theta$. The general case can be proved similarly. It is also a consequence of  $\mathcal{T}(K_\theta, K_\alpha)$ being shift invariant, as we will see below.

For the proof of transitivity assume firstly that a rank-one operator $f\otimes g$ annihilates all operators from $\mathcal{T}(K_\theta, K_\alpha)$, where $f\in\kdt$, $g\in\kda$. Then for all $\varphi\in L^\infty$ we have
\begin{equation}
0=\langle \ata_\varphi f,g\rangle=\langle P_\alpha (\varphi f),g\rangle=\langle \varphi f,g\rangle=\int_{\mathbb{T}}\varphi f\bar gdm.
\end{equation}
Hence $f\bar g=0$ and $f\bar g\in L^1$. Since $f,g\in H^2$, then one of them has to be zero. Since there is no rank-one operator in the preanihilator of $\mathcal{T}(K_\theta, K_\alpha)$, the space $\mathcal{T}(K_\theta, K_\alpha)$ is transitive.

 To prove $2$--reflexivity let us take $f\in\kdt$, $g\in\kda$ such that $zf\in\kdt$, $zg\in\kda$. Since  operators from $\mathcal{T}(K_\theta, K_\alpha)$ are shift invariant, we know that $f\otimes g-zf\otimes zg\in\mathcal{T}(K_\theta, K_\alpha)_\bot$. On the other hand, if for a bounded operator $A\in \mathcal{B}(\kdt,\kda)$ we have
$$0=<A,f\otimes g-zf\otimes zg>=\langle Af,g\rangle-\langle A(zf),zg\rangle, $$ then
by Theorem \ref{sh-inv}, $A\in\mathcal{T}(K_\theta, K_\alpha)$. Hence we obtained $2$--reflexivity.
\end{proof}


\begin{theorem}\label{th53}
The subspace $\mathcal{T}^2(K_\theta^\perp,K_\alpha^\perp)$ is weak*(WOT)--closed, transitive  and $2$--reflexive.
\end{theorem}

\begin{corollary}
The subspace $\mathcal{T}^2(K_\theta^\perp)$ is weak*(WOT)--closed, transitive  and $2$--reflexive.
\end{corollary}
\begin{proof}[Proof of Theorem \ref{th53}]
Let us denote
\begin{equation}
\mathcal{M}=\Big\{zf\otimes z g-f\otimes g\colon\  f\in K_\theta^\perp\cap\{\bar z\}^\perp, g\in K_\alpha^\perp\cap\{\bar z\}^\perp \Big\}.
\end{equation}
Note firstly that an operator $D\in \mathcal{B}(K_\theta^\perp,K_\alpha^\perp)$ is shift invariant if and only if
\[ \langle D, t \rangle=0\quad\text{ for all}\quad t\in \mathcal{M} .\]
On the other hand, Theorem \ref{shinth} means exactly that
\begin{equation}\label{twoperp}
\Big(\mathcal{T}^2(K_\theta^\perp,K_\alpha^\perp)\Big)_\perp\supset\mathcal{M}\quad\quad\text{and }\quad\quad\mathcal{M}^\perp\subset
\mathcal{T}^2(K_\theta^\perp,K_\alpha^\perp).
\end{equation}
Hence the space $\mathcal{T}^2(K_\theta^\perp, K_\alpha^\perp)$ is weak${}^*$--closed (WOT--closed), since it is characterized by annihilating some trace class (finite rank) operators. Moreover, $\mathcal{T}^2(K_\theta^\perp, K_\alpha^\perp)$ is in fact $2$--reflexive
because
\begin{equation*}
\Big(\Big(\mathcal{T}^2(K_\theta^\perp,K_\alpha^\perp)\Big)_\perp \cap \mathcal{F}_2\Big)^\perp\subset\Big(\mathcal{M}\cap\mathcal{F}_2\Big)^\perp=\mathcal{M}^\perp\subset\mathcal{T}^2(K_\theta^\perp,K_\alpha^\perp).
\end{equation*}

For the proof of transitivity let $f=\bar z \bar f_-+\theta f_+$, $g=\bar z \bar g_-+\alpha g_+$, $  f_-, f_+, g_-, g_+\in H^2$, and $f\otimes g\in \Big(\mathcal{T}^2(K_\theta^\perp,K_\alpha^\perp)\Big)_\perp$. In particular, for all $\varphi_1\in L^\infty$,
\[ 0=\langle \hat{T}^{\theta,\alpha}_{\varphi_1}f_+,g_+ \rangle=\int_{\mathbb{T}} \varphi_1 f_+\bar g_+\,dm. \]
Hence $L^1\ni f_+\bar g_+=0$, and as a consequence $f_+=0$ or $g_+=0$, since  $ f_+, g_+\in H^2$. 
Similarly we can show that $f_-=0$ or $g_-=0$. Assume that $f_+\not=0$ a.e. and $g_-\not=0$ a.e. on $\mathbb{T}$. Then, for any $\varphi_3\in L^\infty$, we have
\begin{equation*} 0=\langle \hat{\Gamma}^\theta_{\varphi_3}\theta f_+,\bar z \bar g_- \rangle=\langle P^-{\varphi_3}\theta f_+,\bar z \bar g_- \rangle=\langle {\varphi_3}\theta f_+,\bar z \bar g_- \rangle=\int_{\mathbb{T}} \varphi_3\theta f_+ z  g_-\,dm. \end{equation*}
This implies that $\theta f_+ z g_-= 0$ a.e. on $\mathbb{T}$, which is a contradiction.
\end{proof}

\begin{theorem}\label{tworef}
Let $\theta$, $ \alpha$ be  nonconstant  inner functions. Then the space $\mathcal{T}(K_\theta^\perp, K_\alpha^\perp )$ of asymmetric dual truncated Toeplitz operators is weak*(WOT)--closed and $2$--reflexive.
\end{theorem}
\begin{corollary}\label{tworef1}
Let $\theta$ be a nonconstant  inner function. Then the space $\mathcal{T}(K_\theta^\perp)$ of dual truncated Toeplitz operators is weak*(WOT)--closed and $2$--reflexive.
\end{corollary}
Before we prove Theorem \ref{tworef} let us introduce some notations which allow us to formulate Theorem \ref{thm5.7} differently. Denote
\begin{displaymath}
\begin{split}
\mathcal{M}_1&=\Big\{\theta h\otimes\alpha  g- z\theta h\otimes z\alpha g\ \colon\quad h, g\in H^2\Big\},\\
\mathcal{M}_2&=\Big\{\alpha\theta h\otimes \alpha\theta g-\bar z\bar g\otimes\bar z\bar h\ \colon\  h,g\in H^2 \Big\},\\
\mathcal{M}_3&=\Big\{z\theta h\otimes\bar z\bar g- \theta h\otimes\bar z^2\,\bar g\ \colon\  h, g\in H^2\Big\},\\
\mathcal{M}_4&=\Big\{\bar z\bar h \otimes z\alpha g - \bar z^2\,\bar h\otimes\alpha g\ \colon\  h, g\in H^2\Big\},\\
\mathcal{M}_5&=\Big\{\theta\otimes\bar z\bar g- \theta\alpha z g\otimes\alpha\ \colon\   g\in H^2\Big\},\\
\mathcal{M}_6&=\Big\{\theta\otimes\alpha\theta z g-\bar z\bar g\otimes\alpha\ \colon\   g\in H^2\Big\}.
\end{split}
\end{displaymath}



\begin{lemma}\label{lranktwo}	
Let $\theta$, $\alpha$ be  inner functions and let $D\in  \mathcal{B}(K_\theta^\perp,K_\alpha^\perp)$. Then
 \begin{enumerate}
 \item $\daa=\hat{T}_{\overline{z}}^{\alpha}\daa\hat{T}_z^{\theta}$ if and only if $D\in \mathcal{M}_1^\perp$;
    \item $\dd=(P^-C_{\alpha}M_{\bar\theta})_{|\theta H^2}\left(\daa\right)^*
			(M_{\theta}C_{\alpha})_{|H^2_-}$ if and only if $D\in \mathcal{M}_2^\perp$;
    \item $\dc\hat{T}_z^{\theta}=\check{T}_z \dc$ if and only if $D\in \mathcal{M}_3^\perp$;
  \item $(\dba)^*\hat{T}_z^\alpha =\check{T}_z (\dba)^*$ if and only if $D\in \mathcal{M}_4^\perp$;
    \item $P^-(D(\theta))=P^-(\theta\alpha\overline{D^*(\alpha)})$ if and only if $D\in \mathcal{M}_5^\perp$;
    \item $P^-(D^*(\alpha))=P^-(\theta\alpha\overline{D(\theta)})$ if and only if $D\in \mathcal{M}_6^\perp$.
 \end{enumerate}
\end{lemma}

\begin{proof} Condition (1), in some sense, was shown in \cite{KKMP}.
To prove condition (2) take $h,g\in H^2$. Then by \eqref{ctheta},
\begin{align*}
\langle (P^-C_{\alpha}M_{\bar\theta})_{|\theta H^2}&\left(\daa\right)^*
			(M_{\theta}C_{\alpha})_{|H^2_-}\bar z\bar g,\bar z \bar h\rangle-\langle \dd\bar z\bar g,\bar z\bar h\rangle
\\&=
\langle M_{\theta}C_{\alpha} \bar z \bar h,\left(\daa\right)^*
			(M_{\theta}C_{\alpha})_{|H^2_-}\bar z\bar g\rangle-\langle \dd\bar z\bar g,\bar z\bar h\rangle \\
&=\langle \daa({\theta}{\alpha}  h),{\theta}{\alpha} g\rangle-\langle \dd\bar z\bar g,\bar z\bar h\rangle
\\
&=\langle D \alpha\theta h,\alpha\theta g \rangle-\langle D\bar z\bar g,\bar z\bar h\rangle
=<D,\alpha\theta h\otimes \alpha\theta g-\bar z\bar g\otimes\bar z\bar h>.
\end{align*}
Consider condition (3) and take $ h, g\in H^2$. Then
\begin{align*}
\langle \dc\hat{T}_z^{\theta}\theta h, \bar z\bar g\rangle-\langle \check{T}_z &\dc\theta h,\bar z\bar g\rangle \\
&=\langle \dc z\theta h, \bar z\bar g\rangle-\langle \dc \theta h,\bar z^2\bar g\rangle \\&=
<D, z\theta h\otimes \bar z\bar g- \theta h\otimes\bar z^2\bar g>.
\end{align*}
 This implies equivalence in (3).
 Similarly, the equivalence in (4) is a consequence of the fact that for all $g,h\in H^2$ the following holds:
\begin{align*}
\langle \bar z\bar h,(\dba)^*\hat{T}_z^{\alpha}\alpha g\rangle-&\langle \bar z\bar h,\check{T}_z (\dba)^*\alpha g\rangle \\
&=\langle \dba \bar z\bar h ,z\alpha g \rangle-\langle \dba \bar z^2\bar h,\alpha g\rangle
\\&=
\langle D, \bar z\bar h \otimes z\alpha g - \bar z^2\bar h\otimes\alpha g\rangle.
\end{align*}

Now  we will show (5). Note that the equality $P^-(D(\theta))=P^-(\theta\alpha\overline{D^*(\alpha)})$ is equivalent to
\[0=\langle D(\theta), \bar z\bar g\rangle-\langle \theta\alpha\overline{D^*(\alpha)},\bar z\bar g\rangle \quad\text{for all}\quad \bar z\bar g\in H^2_-.\] Observe also that
\begin{multline*}
\langle \theta\alpha\overline{D^*(\alpha)},\bar z\bar g\rangle=\int_{\mathbb{T}} \theta\alpha\overline{D^*(\alpha)} z g\,dm=
\overline{\int_{\mathbb{T}} D^*(\alpha)\bar\theta\bar\alpha\bar  z\bar g\,dm}\\=
\overline{\langle D^*(\alpha),\theta\alpha  z g\rangle}=\langle \theta\alpha  z g, D^*(\alpha)\rangle=\langle D(\theta\alpha  z g), \alpha\rangle.
\end{multline*}
Hence \[0=\langle D(\theta), \bar z \bar g\rangle-\langle D (\theta\alpha z g),\alpha\rangle
=<D, \theta\otimes \bar z\bar g- \theta\alpha z g\otimes\alpha> . \]

To show (6) note firstly that the equality $P^-(D^*(\alpha))=P^-(\theta\alpha\overline{D(\theta)})$ is equivalent to
\[0=\langle D^*(\alpha), \bar z\bar g\rangle-\langle \theta\alpha\overline{D(\theta)},\bar z\bar g\rangle \quad\text{for all}\quad \bar z\bar g\in H^2_-.\] We also have
\begin{multline*}
\langle \theta\alpha\overline{D(\theta)},\bar z\bar g\rangle=\int_{\mathbb{T}} \theta\alpha\overline{D(\theta)} z g\,dm=
\overline{\int_{\mathbb{T}} D(\theta)\bar\theta\bar\alpha\bar  z\bar g\,dm}\\=
\overline{\langle D(\theta),\alpha \theta z g\rangle}=\langle\alpha \theta z g, D(\theta)\rangle.
\end{multline*}Hence \[0=\langle D(\theta), \alpha\theta z g\rangle-\langle D (\bar z\bar g),\alpha\rangle
=<D,\theta\otimes \alpha\theta z g-  \bar z\bar g\otimes\alpha> . \]
\end{proof}

\begin{proof}[Proof of Theorem \ref{tworef}]
Theorem \ref{th8.12} together with Lemma \ref{lranktwo} give
\[\Big(\mathcal{T}(K_\theta^\perp,K_\alpha^\perp)\Big)_\perp\supset \bigcup_{l=1}^6\mathcal{M}_l,\quad \mathcal{T}(K_\theta^\perp,K_\alpha^\perp)\supset\bigcap_{l=1}^6\mathcal{M}_l^\perp=
\Big(\bigcup_{l=1}^6\big(\mathcal{M}_l\big)\Big)^\perp.\]
Hence the space $\mathcal{T}(K_\theta^\perp)$ is weak${}^*$--closed (WOT--closed), since it is characterized by annihilating some trace class (finite rank) operators. Moreover,
\[\big(\mathcal{T}(K_\theta^\perp,K_\alpha^\perp)_\perp\cap\mathcal{F}_2\big)^\perp\subset
\Big(\bigcup_{l=1}^6\big(\mathcal{M}_l\cap\mathcal{F}_2\big)\Big)^\perp=\Big(\bigcup_{l=1}^6\mathcal{M}_l\Big)^\perp\subset \mathcal{T}(K_\theta^\perp,K_\alpha^\perp)\]
which means $2$--reflexivity of
$\mathcal{T}(K_\theta^\perp, K_\alpha^\perp)$.
\end{proof}

\section{Functional calculus for ADTTO and reflexivity of its subspaces}
 The property of reflexivity ($k$--reflexivity) is generally not hereditary for subspaces, however, if an additional condition (property $\ak$) is satisfied, it is hereditary.
 Recall from \cite{BFP} the following
definition.
\begin{definition}\label{propA}
A subspace $\s\subset\bhk$ has {\it property $\ak$}, if $\s$ is
weak*--closed and for any weak*--continuous functional $\Lambda$ on
$\s$ there is $g\in \mathcal{F}_k$ such that  for $S\in\s$ we have 
\mbox{$\Lambda
(S)=<S,g>$}.
Let $r\ge 1$. It is said that $\s$ has {\it property $\ak(r)$},
if $\s$ has property $\ak$ and for any $\varepsilon>0$ an operator
$g\in \mathcal{F}_k$ can be chosen such that $\|g\|_1\le
(r+\varepsilon)\|\Lambda\|$.
\end{definition}
The importance of the definition above in our setting is shown in the following proposition
recalled from  \cite{Larson}.
\begin{proposition}\label{lar}Let  $\s\subset \bhk$ be $k$-reflexive. Then
any weak*--closed subspace ${\s}_1\subset\s$ is $k$--reflexive if and
only if $\s$ has property $\ak$.
\end{proposition}
\begin{theorem}\label{funcaldual}
Let $\Phi\colon L^\infty\to \mathcal{T}(K_\theta^\perp, K_\alpha^\perp)\subset \mathcal{B}(K_\theta^\perp, K_\alpha^\perp)$, $\varphi\mapsto D_\varphi^{\theta,\alpha}$. Then
\begin{enumerate}
\item $\Phi $ is a linear isometry onto $\mathcal{T}(K_\theta^\perp, K_\alpha^\perp)$,
  \item $\Phi$ is  continuous in weak${}^*$ topology in both spaces,
  \item $\Phi_{|H^\infty}$ is a bijection onto $\mathcal{A}(K_\theta^\perp, K_\alpha^\perp)$,
\item $\mathcal{T}(K_\theta^\perp,K_\alpha^\perp)$ has property $\mathbb{A}_{1/2}(1)$,
\item any weak${}^*$ closed subspace $\s\subset \mathcal{T}(K_\theta^\perp,K_\alpha^\perp)$ is $2$--reflexive,
\item $\mathcal{A}(K_\theta^\perp, K_\alpha^\perp)$ is weak* closed and $2$--reflexive.
\end{enumerate}
\end{theorem}
\begin{proof} The linearity of $\Phi$ is trivial. The fact that $\Phi$ is an isometry was shown in \cite[Proposition 1.1]{CKLPd}. To see the continuity
take $t=\sum_{k=0}^{\infty} f_k\otimes g_k$, $f_k\in K_\theta^\perp$, $g_k\in K_\alpha^\perp$. Assume that $\varphi_\iota\overset{w*}{\to} \varphi$, then
$$< D_{\varphi_\iota}^{\theta,\alpha},t >=\sum_{k=0}^{\infty} \langle D_{\varphi_\iota}^{\theta,\alpha}f_k,g_k\rangle=\sum_{k=0}^{\infty}\int_{\mathbb{T}} {\varphi_\iota} f_k\bar g_k dm=\int_{\mathbb{T}} \varphi_\iota \sum_{k=0}^{\infty} f_k\bar g_k dm.$$
Since $\sum_{k=0}^{\infty} f_k\bar g_k\in L^1$, then \[\int_{\mathbb{T}} \varphi_\iota \sum_{k=0}^{\infty} f_k\bar g_k dm\to \int_{\mathbb{T}} \varphi \sum_{k=0}^{\infty} f_k\bar g_k dm=<D_\varphi^{\theta, \alpha},t>.\]

To prove (4) let $\Lambda$ be a weak*--continuous functional on $\mathcal{T}(K_\theta^\perp, K_\alpha^\perp)$. Then $\Lambda\circ\Phi\colon L^\infty\to \mathbb{C}$ is a weak*--continuous functional. Hence there is  $f\in L^1$ such that
$$\Lambda(D^{\theta,\alpha}_\psi)=(\Lambda\circ\Phi)(\psi)=\int_{\mathbb{T}} \psi f dm, \quad \psi\in L^\infty.$$
Recall that $H^2\cdot\overline{H^2}$ is a proper subset of $L^1$.
By  \cite{Bourgain} and \cite[Corollary 3.3]{AP}, for any $\varepsilon>0$ there are $h_1,h_2,g_1, g_2\in H^2$ such that $f=h_1\bar h_2+g_1\bar g_2$ 
and $\|g_1\| \|\bar g_2\|<\varepsilon$. Take $t=\theta\alpha h_1\otimes\alpha\theta h_2+\theta\alpha g_1\otimes\alpha\theta g_2\in \mathcal{B}_1(K_\theta^\perp,K_\alpha^\perp)$. Then for $\psi\in L^\infty$ we have
\begin{multline*}<D_\psi^{\theta,\alpha},t>=<D^{\theta,\alpha}_\psi,\theta\alpha h_1\otimes\alpha\theta h_2+\theta\alpha g_1\otimes\alpha\theta g_2>\\=\int_{\mathbb{T}} \psi(\theta \alpha h_1\overline{\theta\alpha h_2}+\theta \alpha g_1\overline{\theta\alpha g_2})dm=\int_{\mathbb{T}} \psi fdm=\Lambda(D^{\theta,\alpha}_\psi).\end{multline*}
Moreover, by \cite[Proposition 1.1]{CKLPd},
$\|\Lambda\|=\|\Lambda\circ\Phi\|=\|f\|_1 $ and $$\|t\|\leqslant \|\theta \alpha g_1\overline{\theta\alpha g_2}\|+\|\theta\alpha h_1\overline{\theta\alpha h_2}\|\leqslant \varepsilon+\|f\|_1+\varepsilon \leqslant \|f\|_1(1+\varepsilon_1).$$
Hence (4) follows.

 Theorem \ref{tworef} and (4) together with Proposition \ref{lar} give us (5). To prove (6) we need to show  that
 $\mathcal{A}(K_\theta^\perp, K_\alpha^\perp)$ is weak*--closed.  Remark \ref{danal} shows that an operator $D\in \mathcal{T}(K_\theta^\perp, K_\alpha^\perp)$  belongs to $\mathcal{A}(K_\theta^\perp, K_\alpha^\perp)$ if and only if  $P^- (zD(\bar z))=0$.
  This in turn is equivalent to the fact that for every $f_-\in H^2_-$,
  \[0=\langle P^- (zD(\bar z)),f_-\rangle=\langle D(\bar z),\bar zf_-\rangle=< D,\bar z\otimes\bar zf_->.\]
  Hence $\mathcal{A}(K_\theta^\perp, K_\alpha^\perp)$ is weak*--closed as a kernel of a weak* continuous functional.

\end{proof}

\end{document}